\documentclass[preprint,12pt]{elsarticle}

\usepackage{amssymb}
\usepackage[OT2,OT1]{fontenc}
\usepackage[latin1]{inputenc}
\usepackage{amsmath}
\usepackage{amsthm}
\usepackage{amscd}
\usepackage{pb-diagram, pb-xy}
\usepackage[all]{xy}
\usepackage{hyperref}
\usepackage{MnSymbol}
\hypersetup{colorlinks=true}

\newtheorem{thm}{Theorem}[section]
\newtheorem{cor}[thm]{Corollary}

\newtheorem{lem}[thm]{Lemma}

\theoremstyle{definition}
\newtheorem{defn}[thm]{Definition}
\newtheorem{rem}[thm]{Remark}
\newtheorem{que}[thm]{Question}

\newtheorem{para}[thm]{--}

\newcommand{\IC} {\mathbb{C}}
\newcommand{\IQ} {\mathbb{Q}}
\newcommand{\IR} {\mathbb{R}}
\newcommand{\IZ} {\mathbb{Z}}
\newcommand{\cO} {\mathcal{O}}
\DeclareMathOperator{\GL}{GL}

\newcommand{\R} {\ensuremath{\mathbb{R}}}

\newcommand{\Q} {\ensuremath{\mathbb{Q}}}
\newcommand{\Z} {\ensuremath{\mathbb{Z}}}

\newcommand{\rank} {\ensuremath{\textnormal{rank}}}
\newcommand{\OO} {\ensuremath{\mathcal{O}}}

\newcommand{\norm}[1]{\Vert#1\Vert}
\newcommand{\abs}[1]{\vert#1\vert}
\newcommand{\angl}[1]{\left\langle #1\right\rangle}
\newcommand{\tq}{\: | \:}
\newcommand{\vide}{\varnothing}

\makeatletter
\def\ps@pprintTitle{%
	\let\@oddhead\@empty
	\let\@evenhead\@empty
	\let\@oddfoot\@empty
	\let\@evenfoot\@oddfoot
}
\makeatother

\begin{document}

\begin{frontmatter}

\title{Inhomogeneous minima of mixed signature lattices}

\author{Eva Bayer-Fluckiger\fnref{label1}}
\ead{eva.bayer@epfl.ch}

\author{Martino Borello\fnref{label1}}
\ead{martino.borello@univ-paris8.fr}

\author{Peter Jossen\fnref{label2}}
\ead{peter.jossen@kcl.ac.uk}

\fntext[label1]{{\sc \'{E}cole Polytechnique F\'{e}d\'{e}rale de
Lausanne}, SB MathGeom CSAG, B\^{a}timent MA, Station 8, CH-1015
Lausanne, Switzerland}

\fntext[label2]{{\sc Eidgen\"{o}ssische Technische Hochschule
Z\"{u}rich}, Gruppe 4, Departement Mathematik, HG G 66.2,
R\"{a}mistrasse 101, CH-8092 Z\"{u}rich, Switzerland}

\begin{abstract}
We establish an explicit upper bound for the Euclidean minimum of a
number field which depends, in a precise manner, only on its
discriminant and the number of real and complex embeddings. Such
bounds were shown to exist by Davenport and Swinnerton-Dyer
(\cite{MR0047707,MR0056644,MR0075988}). In the case of totally real
fields, an optimal bound was conjectured by Minkowski and it is
proved for fields of small degree. In this note we develop methods
of McMullen (\cite{MR2138142}) in the case of mixed signature in
order to get explicit bounds for the Euclidean minimum.
\end{abstract}

\begin{keyword}
Lattices \sep Euclidean minimum

\MSC 80.010 \sep 80.030
\end{keyword}

\end{frontmatter}

\section{Introduction}

\begin{par}
Let $K$ be a number field of degree $n$, let $\OO_K$ be its ring of
integers and $d_K$ be the absolute value of its discriminant. Let
${N}: K \to \IQ$ be the absolute value of the norm map. The number
field $K$ is said to be {\it Euclidean with respect to the norm} if
for every $a,b \in \OO_K$ with $b \not = 0$ there exist $c, d \in
\OO_K$ such that $a = bc + d$ and ${ N}(d) < { N}(b)$. Equivalently,
the number field $K$ is Euclidean with respect to the norm if for
every $x \in K$ there exists $c \in \OO_K$ such that ${ N}(x-c) <
1$. This suggests to look at the real number
\begin{equation}\label{Eqn:DefOfEuclideanMinForNumberField}
M(K) := {\rm sup}_{x \in K} {\rm inf}_{c \in \OO_K} { N}(x-c),
\end{equation}
called the \emph{Euclidean minimum} of $K$. If $K$ is totally real,
then Minkowski's conjecture states that the inequality
\begin{equation}\label{Eqn:MinkowskiConj}
M(K) \le 2^{-n}\cdot \sqrt {d_K}
\end{equation}
holds. This conjecture is known for $n \le 9$ (see \S
\ref{Par:MinkowskiConjStatus} for details). It is natural to look
for bounds similar to (\ref{Eqn:MinkowskiConj}) for number fields of
mixed signature. Such bounds were obtained by Clarke
\cite{MR0047706} and Davenport \cite{MR0047707}. The latter proved
that for every pair of nonnegative integers $(r,s)$ there exists a
constant $C_{r,s}$ such that
\begin{equation}\label{Eqn:DavenportInequality}
M(K) \leq 2^{\frac{-sn}{r+s}}\cdot C_{r,s}\cdot {d_K}^{\frac{n}{2(r+
s)}}
\end{equation}
holds for every number field $K$ of signature $(r,s)$. Although an
explicit constant $C_{r,s}$  can be deduced from Davenport's proof,
it is too large to be useful. The aim of this paper is to develop
methods of McMullen (cf. \cite{MR2138142}) to obtain a better
constant. A weakened but easy to read form of our main result
(Theorem \ref{thm:main}) is the following:
\end{par}

\vspace{4mm}
\begin{par}{\bf Theorem.}\emph{
Let $K$ be a number field of signature $(r,s)$ and degree $n =
r+2s\geq 4$, and let $d_K$ be the absolute value of the discriminant
of $K$. The following inequality holds.
$$M(K)\leq 2^{\frac{-sn}{r+s}}\cdot \big ( \tfrac 12 \sqrt n\big )^n  \cdot {d_K}^{\frac{n}{2(r+s)}}.$$}
\end{par}

\begin{par}
In other words, the inequality \eqref{Eqn:DavenportInequality} holds
with the constant $C_{r,s} = 2^{-n} \cdot n^{\frac n 2}$, which is
still very large. In \S \ref{sec1} we compare our result to known
estimates for totally real and totally imaginary fields, and
formulate some questions as to which bounds one might hope for.
\end{par}

\vspace{4mm}
\begin{par}{\bf Convention:}
Throughout the paper, all vector spaces are understood to be finite
dimensional real vector spaces. A \emph{scalar product} on a vector
space $V$ is a symmetric positive definite bilinear form $V\times V
\to \R$, and by a \emph{lattice} in $V$ we understand a cocompact
discrete subgroup of $V$.
\end{par}

\vspace{14mm}
\section{Inhomogeneous minima of lattices of mixed signature}\label{sec1}

\begin{par}
We explain in this section what we mean by lattices of mixed
signature, recall the definition and some properties of
inhomogeneous minima, and give some motivation for studying these
objects.
\end{par}

\vspace{4mm}
\begin{para}\label{Para:IntroLatticeOfSignature}
\begin{par}
Let $r$ and $s$ be nonnegative integers and set $n=r+2s$. Let
$V:=\R^r\oplus \IC^s$, which is an $n$--dimensional vector space over
$\R$. We equip $V$ with the scalar product $\angl{-,-}$ given by
\begin{equation}\label{Eqn:ScalarProdInV}
\angl{(v_1,\ldots,v_{r+s}),(v_1',\ldots,v_{r+s}')} = \sum_{i=1}^r
v_iv'_i + \sum_{i=r+1}^{r+s} {\rm Re}(v_i\overline v_i')
\end{equation}
and we call the function $N:V\rightarrow \R$ given by
\begin{equation}\label{eq}
N(v_1,\ldots,v_{r+s})=  | v_1 \cdot \ldots \cdot v_r\cdot
v_{r+1}^2\cdot \ldots \cdot v_{r+s}^2|
\end{equation}
the \emph{norm}, although it is not a norm in the usual sense.
\end{par}
\begin{par}
By a \emph{lattice of signature $(r,s)$} we mean a lattice in this
particular vector space $V$. We denote by $\det(\Lambda)$ the volume
of $V/\Lambda$ with respect to the volume form obtained from the
scalar product (\ref{Eqn:ScalarProdInV}). We call the real numbers
$$m(\Lambda):=\inf_{\lambda\in \Lambda\setminus \{0\}}N(\lambda)  \qquad  \qquad \text{and}  \qquad  \qquad  \ M(\Lambda):=\sup_{v\in V}\inf_{\lambda\in\Lambda}N(v-\lambda)$$
the \textit{homogeneous minimum}, respectively the
\textit{inhomogeneous minimum}, of $\Lambda$.
\end{par}
\end{para}

\vspace{4mm}
\begin{para}
Our motivation for studying inhomogeneous minima of lattices of
mixed signature is the classical geometry of numbers. Let $K$ be a
number field of degree $n$ and signature $(r,s)$, so that $n=r+2s$.
Let $d_K$ be the absolute value of the discriminant of $K$. Choosing
an ordering of the $r$ real and of $s$ non-conjugated complex
embeddings of $K$ we obtain a $\IQ$--linear embedding $K\to V$ with
dense image. The image of the ring of integers $\OO_K$ of $K$ in $V$
is a lattice of volume $2^{-s}\cdot \sqrt{d_K}$ and the norm map
$N:V\rightarrow \R$ given in \eqref{eq} continuously extends the
absolute value of the usual norm map ${ N}:K\rightarrow \Q$, hence
the name. In this context, $m(K)$ and $M(K)$ denote the homogeneous
and the inhomogeneous minimum of $\OO_K$ as a lattice in $V$. These
quantities are independent of the ordering of the real and complex
embeddings of $K$. Cerri proves in \cite{MR2222729} that $M(K)$ is
equal to the Euclidean minimum of the number field $K$ as given in
\eqref{Eqn:DefOfEuclideanMinForNumberField}.
\end{para}

\vspace{4mm}
\begin{para}\label{Par:MinkowskiConjStatus}
\emph{Minkowski's conjecture} on inhomogeneous minima of products of
real linear forms (see for instance \cite{MR893813}) states that if
$s=0$, the inequality
$$M(\Lambda)\leq 2^{-n}\cdot \det(\Lambda)$$
holds for every lattice $\Lambda$ in $V$. The conjecture is proved
for $n$ up to $9$  (\cite{MR1511108,MR1544627,MR0025515,
MR0306128,MR2138142, MR2516969,MR2776094, LR14}). In terms of number
fields, Minkowski's conjecture implies that
$$M(K)\leq 2^{-n} \cdot \sqrt{d_K}$$
holds for every totally real number field $K$. This is proved also
for particular totally real fields of degree $n> 9$, see for example
\cite{MR2211300, MR2274907, BSUA, MR3118625}. For totally real
number fields of any degree $n$, the inequality
\begin{equation}\label{Eqn:Cebotarev}
M(K)\leq (\sqrt 2)^{-n} \cdot \sqrt{d_K}
\end{equation}
holds by a theorem of Chebotarev, see for instance \cite{Dav46}.
There are improvements of this estimate, yet, to our best knowledge,
it is at present not known whether $M(K)\leq c^{-n} \cdot
\sqrt{d_K}$ holds for some real number $c>\sqrt 2$.
\end{para}

\vspace{4mm}
\begin{para}
Improving earlier results of  Clarke (\cite{MR0047706}), Davenport
shows in \cite{MR0047707} that for every pair $(r,s)$ there exists a
real number $C_{r,s}\geq 0$ such that the inequality (with the
convention $0^0=1$)
\begin{equation}\label{Eqn:DavenportIneqForLattice}
m(\Lambda)^\frac{s}{r+s}\cdot M(\Lambda) \leq C_{r,s}\cdot
\det(\Lambda)^\frac n{r+s}
\end{equation}
holds for all lattices of signature $(r,s)$. Notice that if we scale
$\Lambda$ by a real number $t > 0$, both sides of the inequality
change by the factor $t^{n^2/(r+s)}$. Also notice that the
inequality is trivial if $m(\Lambda)=0$, unless $s=0$. In terms of
number fields, where we have $m(\Lambda)=1$, Davenport's result
states that
\begin{equation}\label{Eqn:DavenportIneqForField}
M(K)\leq 2^{-\frac{sn}{r+s}} \cdot C_{r,s}\cdot
{d_K}^{\frac{n}{2(r+s)}}.
\end{equation}
holds for every number field $K$ of signature $(r,s)$. In the case
$s=0$ one hopes that the inequality
\eqref{Eqn:DavenportIneqForLattice} holds with the constant $C_{r,0}
= 2^{-r}$, according to Minkowski's conjecture, and one knows
\eqref{Eqn:DavenportIneqForLattice} to hold for $C_{r,0} = (\sqrt
2)^{-r}$, according to Chebotarev's theorem. In \cite{MR2274907} it
is proved that
\begin{equation}\label{Eqn:Bayer}
M(K)\leq 2^{-n}\cdot d_K
\end{equation}
holds for any number field $K$ of degree $n$. In other words, for
$r=0$ the inequality \eqref{Eqn:DavenportIneqForField} holds with
the constant $C_{0,s} = 1$.
\end{para}

\vspace{4mm}
\begin{para}
We wonder what an analogue of Minkowski's conjecture for number
fields or lattices of mixed signature should look like. In an
updated version of \cite{MR1362867} of 2004, Lemmermeyer states that
``similar results [to Minkowksi's conjecture] (not even a
conjecture) for fields with mixed signature are not known except for
a theorem of Swinnerton-Dyer (\cite{MR0065592}) concerning cubic
fields''. The result in question states that the inequality
\begin{equation}\label{cubic}
M(K)\leq \frac{1}{16\sqrt[3]{2}}\cdot {d_K}^{\frac{2}{3}}
\end{equation}
holds when $K$ is a complex cubic field. Also, the Euclidean minimum
of any complex quadratic field $K$ is known, and we have that
$M(K)\leq \frac{1}{8}\cdot d_K$.
\end{para}

\vspace{4mm}
\begin{para}\label{Par:MixedMinkowski1}
If we agree that a mixed signature analogue of Minkowski's
conjecture is an estimate of the form
\eqref{Eqn:DavenportIneqForLattice}, then the quantity we search to
determine is the real number
$$c_{r,s} := \sup\big\{m(\Lambda)^\frac s{r+s} \cdot M(\Lambda)\cdot \det(\Lambda)^{\frac{-n}{r+s}}\:\big|\: \Lambda \mbox{ is a lattice of signature }(r,s)\big \} $$
which exists by Davenport's result. Minkowski's conjecture states
$c_{r,0} = 2^{-r}$, but moreover it is part of the conjecture that
the supremum is attained precisely for those lattices $\Lambda
\subseteq \IR^r$ which are sums of $r$ lattices $\Lambda_i = a_i\IZ
\subseteq \IR$. For instance, the standard lattice $\IZ^r\subseteq
\IR^r$ satisfies $M(\IZ^r)= 2^{-r}$. In the case of mixed signature
$(r,s)$, lattices which are sums of $r$ lattices $a_i\IZ \subseteq
\IR$ and $s$ lattices $x_i\IZ + y_i\IZ \subseteq \IC$ are not
helpful in guessing $c_{r,s}$, because their homogeneous minima are
zero. Indeed, the equality
$$c_{r,s} = \sup\big\{M(\Lambda)\cdot \det(\Lambda)^{\frac{-n}{r+s}}\:\big|\: \Lambda \mbox{ is a lattice of signature } (r,s) \mbox{ and } m(\Lambda)=1\big \}$$
holds for $s>0$. It is not clear whether this is true as well when
$s=0$, except in the case $(r,s)=(2,0)$ because there are real
quadratic number fields whose euclidean minimum is arbitrarily close
to Minkowski's bound. Also, inhomogeneous minima of complex
quadratic fields are known, and one can show that $c_{0,1}=\frac12$.
\end{para}

\vspace{4mm}
\begin{que}\label{Que:MixedMinkowski1}
Are there positive real numbers $a$, $b$ such that $c_{r,s}\leq
a^rb^s$ holds for all $(r,s)$, that is, such that the inequality
$$M(\Lambda)  \leq a^rb^s\cdot\det(\Lambda)^{\frac{n}{r+s}}$$
holds for every lattice $\Lambda$ of signature $(r,s)$ and
inhomogeneous minimum $m(\Lambda)=1$? Can one choose $a=(\sqrt
2)^{-1}$ and $b=1$ as suggested by Chebotarev's and estimates
\eqref{Eqn:Cebotarev} and \eqref{Eqn:Bayer}? Could one even choose
$a=b=\frac12$?
\end{que}

\vspace{4mm}
\begin{para}\label{Par:MixedMinkowski2}
The inequality \eqref{cubic} does not fit in the discussion of
\S\ref{Par:MixedMinkowski1}, because the exponent of the
discriminant is not what we were asking for. This is not an accident
that just happens in the case of cubic fields, indeed, Davenport and
Swinnerton-Dyer have shown (cf. Theorem 1 in \cite{MR0075988},
together with \cite{MR0056644,MR0065592}) that there exists a real
number $B_{r,s}>0$ such that
$$ M(\Lambda)\leq B_{r,s}\cdot \det(\Lambda)^{\max(\frac{n-1}{r+s},\frac{n-s}{(r+s)-s/2})} $$
holds for every lattice $\Lambda\subseteq \IR^r\oplus\IC^s$
satisfying $m(\Lambda) = 1$.
\end{para}

\vspace{4mm}
\begin{que}\label{Que:MixedMinkowski2}
Let $A_{r,s} \subseteq \IR$ be the set of those real numbers $\alpha
\in \IR$ with the property that there exists $C \in \IR$ such that
$M(\Lambda)\leq C\det(\Lambda)^\alpha$ holds for every lattice
$\Lambda$ of signature $(r,s)$ and $m(\Lambda)=1$. This set is a
either a closed or an open half line
$$A_{r,s} = (\alpha_{r,s},\infty) \qquad \mbox{or}\qquad A_{r,s} = [\alpha_{r,s},\infty)$$
and we can define constants $c_{r,s}(\alpha)$ analogous to those in
\ref{Par:MixedMinkowski1} and ask questions like those in
\ref{Que:MixedMinkowski1} for any exponent $\alpha\in A_{r,s}$. As
proven by Davenport and Swinnerton-Dyer in
\cite{MR0075988,MR0056644,MR0065592} we have
$$\alpha_{r,s} \leq \max\Big(\frac{n-1}{r+s},\frac{n-s}{(r+s)-s/2}\Big)$$
and this inequality is an equality if $r=0$ (so $\alpha_{0,s}=2$)
and if $r=s=1$ (so $\alpha_{1,1}= \frac 43$). Davenport and
Swinnerton-Dyer suggest also that we might have equality whenever $s
= 1$, and also that the inequality is strict in other cases. What is
the number $\alpha_{r,s}$? Do we have $\alpha_{r,s}\in A_{r,s}$?
\end{que}

\vspace{14mm}
\section{Results on successive minima}\label{sec2}

\begin{par}
We fix for this section a pair of nonnegative integers $(r,s)$, and
denote by $V=\R^r \oplus \IC^{s}$ the vector space of dimension $n =
r+2s$ introduced in \ref{Para:IntroLatticeOfSignature}, equipped
with its scalar product \eqref{Eqn:ScalarProdInV} and norm map
\eqref{eq}. We set $\norm v := \sqrt{\langle v,v\rangle}$. We will
recall the definition and some results concerning the successive
minima of a lattice.
\end{par}

\vspace{4mm}
\begin{defn}
The \textit{successive minima} of a lattice $\Lambda$ in $V$ are the
real numbers $\mu_1(\Lambda),\ldots,\mu_n(\Lambda)$ defined in the
following way: for $m\in\{1,\ldots,n\}$, $\mu_m(\Lambda)$ is the
infimum of the real numbers $r$ such that there exist $m$
independent vectors $\lambda_1,\ldots,\lambda_m$ in
$\Lambda\setminus\{0\}$ with $\norm {\lambda_i} \leq r$ for $i\in
\{1,\ldots,m\}$. Vectors $\lambda\in \Lambda$ with
$\|\lambda\|=\mu_1(\Lambda)$ are called \textit{shortest}
(\textit{nonzero}) \textit{vectors}.
\end{defn}

\vspace{4mm}
\begin{defn}
The \textit{Hermite constant} for dimension $n$ is
$\gamma_n:=\sup_\Lambda \mu_1(\Lambda)^2\cdot \det(\Lambda)^{-2/n}$,
where $\Lambda$ runs over all lattices in $V$.
\end{defn}

\vspace{4mm}
\begin{lem}[Minkowski]\label{Mink}
For every lattice $\Lambda$ in $V$ and $1\leq t\leq n$ we have
$$\mu_1(\Lambda)\cdot \cdots\cdot \mu_t(\Lambda)\leq
\gamma_n^{t/2}\cdot \det(\Lambda)^{t/n},$$ where $\gamma_n$ is the
Hermite constant.
\end{lem}

\begin{proof}
This is Theorem 2.6.8 of \cite{MR1957723}.
\end{proof}

\vspace{4mm}
\begin{lem}\label{lemma-bound-hom}
For every lattice $\Lambda$ in $V$ the following inequality holds:
$$m(\Lambda)\leq \left(\frac{\sqrt
2}{\sqrt{n}}\cdot \mu_1( \Lambda )\right)^n.$$
\end{lem}

\begin{proof}
For every element $v = (v_1,\ldots,v_{r+s})$ of $V$ we have
$$\sum_{i=1}^r v_i^2 + \sum_{i=r+1}^{r+s} \abs{v_i}^2\:\: \leq \:\: 2\norm{v}^2 \:\:=\:\: 2\sum_{i=1}^r v_i^2 + \sum_{i=r+1}^{r+s} \abs{v_i}^2$$
hence
\begin{equation}\label{in-g-a}
N(v)^{\frac 1n} \leq \frac{\sqrt 2}{\sqrt n}\cdot \norm{v}
\end{equation}
by the inequality between arithmetic and geometric means. The
desired inequality follows from
$$m(\Lambda)^{\frac 1 n}\:\:=\:\: \inf_{\lambda \neq 0}N(\lambda)^{\frac 1 n}
\:\:\leq\:\: \frac{\sqrt 2}{\sqrt n}\cdot \inf_{\lambda\neq 0}
\norm{\lambda} \:\:=\:\: \frac{\sqrt 2}{\sqrt{n}}\cdot
\mu_1(\Lambda)$$ where the infima are taken over all non-zero
elements $\lambda$ of $\Lambda$.
\end{proof}

\vspace{4mm}
\begin{lem}\label{lemma-bound-inhom}
The following inequality holds for every lattice $\Lambda$ in $V$:
$$M(\Lambda)\:\: \leq\:\: \left(\frac{1}{\sqrt 2}\cdot\mu_n( \Lambda)\right)^n.$$
\end{lem}

\begin{proof}
Let $v\in V$. By definition of the successive minima, there exist
linearly independent elements $\lambda_1,\ldots,\lambda_n$ of
$\Lambda$ satisfying $\norm{\lambda_i}\leq \mu_n(\Lambda)$ for all
$i\in\{1,\ldots,n\}$. Choose an orthonormal basis $e_1,\ldots,e_n$
of $V$ such that each $\lambda_i$ is written as
$$\lambda_i\:\:=\:\:\sum_{j=1}^ia_{ij}e_j\:\:,$$
in other words, the matrix $(a_{ij})$ is upper triangular. Write
$v=b_1e_1+\cdots+b_ne_n$ and successively choose integers
$k_1,\ldots,k_n$ such that the coefficients $b_j'$ in
$$v-(k_1\lambda_1+\cdots+k_n\lambda_n)\:\:=\:\:\sum_{j=1}^n b_j'e_j$$
satisfy $|b_j'|\leq \frac{1}{2}|a_{jj}|$. Because of
$\|\lambda_i\|\leq \mu_n(\Lambda)$ we have
 $|a_{ij}|\leq \mu_n(\Lambda)$ and hence $|b_j'|\leq \frac{1}{2} \mu_n(\Lambda)$. Setting $\lambda=k_1\lambda_1+\cdots+k_n\lambda_n$, we
obtain the inequality
$$\|v-\lambda\|^2\:\:=\:\:\sum_{j=1}^n |b_j'|^2\:\:\leq\:\: \frac{n}{4}\cdot \mu_n^2(\Lambda)$$
or equivalently $\|v-\lambda\|\leq \frac{\sqrt{n}}{2}\cdot
\mu_n(\Lambda)$. From the inequality between arithmetic and
geometric means, as in \eqref{in-g-a}, we obtain
$$N(v - \lambda) \:\:\leq \:\: \left(\frac{\sqrt 2}{\sqrt n} \cdot\norm{v-\lambda}\right)^n \:\:\leq\:\: \left(\frac{1}{\sqrt 2}\cdot\mu_n( \Lambda)\right)^n\:\:,$$
and since $v$ was arbitrary, the desired inequality follows.
\end{proof}

\vspace{14mm}
\section{McMullen's topological methods}\label{Sec:McMullen}

\begin{par}
A lattice $\Lambda \subseteq \IR^n$ is said to be \emph{well
rounded} if all of its successive minima are equal. A conjecture of
Woods (\cite{MR0302570}) predicts that the \textit{covering radius}
$\rho(\Lambda)$ of a well rounded lattice $\Lambda$ satisfies the
inequality
$$\rho(\Lambda):=\sup_{v\in \IR^n} \inf_{\lambda\in \Lambda} \|v-\lambda\|\leq \frac{\sqrt{n}}{2}\cdot \det(\Lambda)^{\frac{1}{n}}.$$
Covering radii and inhomogeneous minima are linked by the inequality
between arithmetic and geometric means
$$N(v)^{\frac{1}{n}}\leq \frac{\|v\|}{\sqrt{n}},$$
so that, to prove Minkowski's conjecture for $\Lambda \subseteq
\IR^n$ it is sufficient to show that there exists a well rounded
lattice $\Lambda'$ with the same inhomogeneous minimum as $\Lambda$,
and to prove that Woods' conjecture holds for $ \Lambda'$.
\end{par}

\begin{par}
The set of lattices in $\IR^n$ can be identified with the quotient
$\GL_n(\IR)/\GL_n(\IZ)$ and in particular it has the structure of a
differentiable real manifold of dimension $n^2$ with a
differentiable transitive left action by $\GL_n(\IR)$. McMullen
proved in \cite{MR2138142} that if the closure of the orbit of a
lattice $\Lambda \subseteq \IR^n$ under the group of diagonal
matrices with positive entries and determinant 1 is compact, then
some lattice in that closure is well rounded.
\end{par}

\begin{par}
In the case where the orbit of $\Lambda \subseteq \IR^n$ is already
compact, McMullen's result follows with some effort from the
following theorem (Theorem 1.6 of \cite{MR2138142}).
\end{par}

\begin{thm}\label{thm-covering}
Let $T$ be a real torus and let $U_1,\ldots,U_m$ be open subsets of
$T$ which cover $T$. Suppose that the inequality
$$\rank (H_1(W,\Z)\rightarrow H_1(T,\Z))<p$$
holds for every $p\in \{1,\ldots,m\}$ and every connected component
$W$ of $U_p$. Then $m$ is strictly larger than the dimension of $T$.
\end{thm}

\vspace{4mm}
\begin{para}\label{Par:SetupMcMullen}
Let $V=\IR^r \oplus \IC^{s}$ be the vector space introduced in
\ref{Para:IntroLatticeOfSignature}, equipped with its scalar product
and the norm map. Let $G$ be the group of diagonal matrices
$g = {\rm diag}(g_1,\dots,g_{r+s})$ with positive real entries $g_i$
such that $$g_1\dots g_r (g_{r+1}\dots g_{r+s})^2 = 1.$$
The group $G$ acts by $\IR$-linear transformations on $V$ and
preserves the norm. Hence, $G$ acts on the space of lattices in $V$
and preserves homogeneous and inhomogeneous minima.
\end{para}

\vspace{4mm}
\begin{thm}\label{thm-succ-minima}
Let $\Lambda \subseteq V$ be a lattice such that $G\Lambda$ is
compact. Then there exists $g\in G$ such that the equalities
$\mu_{s+1}(g\Lambda) =\cdots = \mu_n(g\Lambda)$ hold.
\end{thm}

\vspace{4mm}
\begin{rem}
For $s=0$, Theorem \ref{thm-succ-minima} is McMullen's Theorem 4.1
in \cite{MR2138142}. With a few adaptations, McMullen's method works
also in our superficially more general setup. For the convenience of
the reader, we check this in detail.  Notice that we cannot expect
to obtain more equalities between successive minima than stated.
This is clear from the case $r=0$. We could alternatively demand for
any $0\leq k \leq s$ the equalites $\mu_{1+k}(g\Lambda)= \cdots =
\mu_{r+s+k}( g\Lambda)$ to hold.
\end{rem}

\vspace{4mm}
\begin{para}
We fix for the rest of this section a lattice $\Lambda \subseteq V
:= \IR^r \oplus \IC^s$ such that $G\Lambda$ is compact. To say that
$G\Lambda$ is compact is to say that the stabiliser
$$G_\Lambda:=\{g \in G \tq g\Lambda=\Lambda\}$$
of $\Lambda \subseteq V$ is cocompact in $G$. Yet in other words,
since $G_\Lambda$ is discrete in $G \simeq \IR^{r+s-1}$, the
quotient $T := G/G_\Lambda$ has to be a real torus of dimension
$r+s-1$. For every $g\in G$, we define:
\begin{eqnarray*}
D(g) & := & \{ \lambda \in \Lambda \tq \norm{g\Lambda} < \mu_n(g\Lambda)\}\\
M(g) & := & \mbox{the real subspace of $V$ generated by $D(g)$}
\end{eqnarray*}
The set $D(g)$ is a finite and nonempty subset of $\Lambda$, and
$M(g) \subseteq V$ is a proper rational subspace, by which we mean
that $M(g)$ is not equal to $V$, and that $M(g)$ is generated by
lattice vectors. Notice that $\dim M(g)$ is the smallest integer
$p\geq 0$ for which $\mu_{p+1}(g\Lambda) = \cdots = \mu_n(g\Lambda)$
holds. Thus, in order to prove Theorem \ref{thm-succ-minima}, we
have to show that there exists a $g\in G$ with $\dim M(g)\leq s$.
\end{para}

\vspace{4mm}
\begin{para}
For every integer $1\leq p \leq n$, let us define $G(p)$ as  the set
of those $g\in G$ with $\dim M(g)=n-p$. The subsets $G(p)$ of $G$
form a partition of $G$, and these sets are stable under the
translation action of $G_\Lambda$. Indeed, for every $g\in G$ and $h
\in G_\Lambda$ we have
$$D(gh)\:\:=\:\:\{\lambda\tq \norm{gh\lambda} < \mu_n(g\Lambda)\} \:\:=\:\: \{h^{-1} \lambda \tq \norm{g\lambda} < \mu_n(g\Lambda)\} \:\:=\:\: h^{-1}D(g)$$
hence $M(gh) = h^{-1}M(g)$ and $\dim M(gh) =  \dim M(g)$. Therefore,
as $p$ ranges over $1,2,\ldots n$, the sets
$$T(p) := G(p)/G_\Lambda \:\:=\:\: \{ gG_\Lambda \tq \dim M(g)=n-p\}$$
form a partition of the torus $T := G/G_\Lambda$. We will need to
know that the subsets $T(p)$ of $T$ are not too pathological, in
particular we want to know that each $T(p)$ admits an open
neighbourhood of which $T(p)$ is a deformation retract.
\end{para}

\vspace{4mm}
\begin{para}
Let us recall some definitions and a theorem from real algebraic
geometry. Let $U$ be a nonepty open subset of some finite
dimensional real vector space.  A subset of $U$ is said to be
\emph{semialgebraic} if it belongs to the  smallest family of
subsets of $U$ which is closed under finite unions, finite
intersections and complements, and which contains the sets
$$\{x\in U \tq f(x)\geq 0\}$$
for polynomial functions $f:U\to \IR$. A central result in real
algebraic geometry is that compact semialgebraic sets admit a finite
triangulation. The following is a light version of Theorem 2.6.12 of
\cite{BR90}:

\begin{thm}\label{Thm:CWDecomposition}
Let $X\subseteq \IR^n$ be a compact semialgebraic set, and let $X_1,
\ldots, X_m$ be a partition of $X$ by semialgebraic subsets. There
exists a finite simplicial complex (a polyhedron) $K$ and a
homeomorphism $h:X\to K$, such that each $X_p$ is a union of sets
$h^{-1}(F)$ for some relatively open faces $F$ of $K$.
\end{thm}

The group $G$ is itself a semialgebraic subset of $\IR^{r+s}$, so we
know what semialgebraic subsets of $G$ are. But we wish to speak
about semialgebraic sets on a real torus, say $T = G/G_\Lambda$. The
real torus $T$ can be given a structure of a compact locally
semialgebraic space (which is: a locally ringed space, locally
isomorphic to a semialgebraic set, see \cite{DK85}, \S I.1,
Definition 3). Theorem \ref{Thm:CWDecomposition} persists in this
generality (\cite{DK85} \S II.4, Theorem 4.4). In elementary terms
however, this means the following: A subset $X$ of $T$ is called
\emph{semialgebraic} if there exist open subsets $U_i\subseteq
\IR^n$ which form an atlas of $T$ via the projection maps
$\varphi_i:U_i\to T$, such that each $\varphi_i^{-1}(X)$ is a
semialgebraic subset of $U_i$.

\begin{cor}\label{Cor:CWDecomposition}
Let $T_1, \ldots, T_m$ be a partition of $T$ by semialgebraic
subsets. There exists a finite simplicial complex $K$ and a
homeomorphism $h:T\to K$, such that each $T_p$ is a union of sets
$h^{-1}(F)$ for some relatively open faces $F$ of $K$.
\end{cor}

We can deduce this corollary either from the triangulation theorem
of \cite{DK85}, but also from the quoted Theorem
\ref{Thm:CWDecomposition}. Indeed, choose for $X\subseteq G$ any
compact semialgebraic set which contains a fundamental domain for
$T$, and for $X_p\subseteq X$ the preimages in $X$ of the subsets
$T_p$. Possibly after subdividing some simplicies, any triangulation
$X\to K$ descends to a triangulation of $T$.
\end{para}

\vspace{4mm}
\begin{lem}\label{Lem:TildeSiLocallyClosed}
The set $T(1) \cup T(2) \cup \cdots\cup T(p)$ is an open,
semialgebraic subset of $T$.
\end{lem}

\begin{proof}
\begin{par}
We start by showing that $G(1) \cup G(2) \cup \cdots\cup G(p)$ is an
open and locally semialgebraic subset of $G$. For every $g\in G$,
let us denote by $E(g) \subseteq V$ the ellipsoid given by the
equation
$$E(g) := \{v\in V \tq \norm{gv} = \mu_n(g\Lambda)\}$$
and notice that $\mu_n(g\Lambda)$ is the largest real number $r$
with the property that the lattice points which lie inside of
$\{v\in V \tq \norm{gv} = r\}$ do not generate $V$ as a real vector
space. The set of lattice points inside of $E(g)$ is indeed the set
$D(g)$. Observe that $\mu_n(g\Lambda)$ is a continuous function of
$g$, and that $E(g)$ varies continuously with $g$, say for the
Hausdorff distance.
\end{par}
\begin{par}
Pick an element $g_0 \in G(p)$. The set of those $g\in G$ with the
property that every lattice point which is on the inside of $E(g_0)$
is also on the inside of  $E(g)$, and every lattice point which is
on the outside of $E(g_0)$ is also on the outside of  $E(g)$, form
an open neighbourhood of $g_0$ in $G$. This neighbourhood, call it
$U$, is contained in
$$G(1) \cup G(2) \cup \cdots\cup G(p) \:\: = \:\: \{g\in G \tq \dim M(g) \geq n-p\}$$
indeed, we have by definition $D(g_0) \subseteq D(g)$ and hence
$M(g_0) \subseteq M(g)$ for all $g\in U$, and thus $U\subseteq
G(1)\cup \cdots \cup G(p)$. This shows that $G(1) \cup G(2) \cup
\cdots\cup G(p)$ is open.
\end{par}
\begin{par}
Next, let us show that $G(q) \cup G(q+1) \cup \cdots \cup G(p)$ is
defined by polynomial inequalities on $U$ for every $1 \leq q \leq
p$. Let $S_0$ be the set of those lattice points which lie on
$E(g_0)$. An element $g\in U$ belongs to $G(q) \cup \cdots \cup
G(p)$ if and only if the following conditions hold:
\begin{enumerate}
\item The set of points $D\subseteq S_0$ which lie inside of $E(g)$ satisfy $\dim \angl{D, M(g_0)} \leq n-q$.
\item The set of points $S \subseteq S_0$ which are on or inside $E(g)$ satisfy $\angl{S, M(g_0)}=V$.
\end{enumerate}
The set of those $g\in U$ satisfying (1) and (2) for given sets
$D\subseteq S$ is described by the quadratic polynomial inequalities
$\norm{g\lambda_0}^2 \geq \norm{g\lambda}^2$ for $\lambda_0 \in
S_0\setminus D$ and $\lambda\in S\setminus D$, which in particular
imply $\norm{g\lambda} =\mu_n(g\Lambda)$ for $\lambda\in S\setminus
D$. This shows that $(G(q)\cup \cdots \cup G(p)) \cap U$ is a finite
union of closed subsets of $U$, each of which is defined by finitely
many polynomial inequalities, namely
$$(G(q)\cup \cdots \cup G(p))\cap U = \bigcup_{D\subseteq S} \big\{g\in U \:\big|\: \norm{g\lambda_0}^2 \geq \norm{g\lambda}^2 \:\: \mbox{for all}\:\: \lambda_0 \in S_0\setminus D, \:\: \lambda \in S\setminus D\big\}$$
where the union runs through all pairs of subsets $D \subseteq S$ of
$S_0$ satisfying $\dim \angl{D, M(g_0)} \leq n-q$ and $\angl{S,
M(g_0)}=V$.
\end{par}
\begin{par}
The quotient map $G\to T$ is locally a polynomial diffeomorphism,
that is a tautology. Hence $T(1) \cup T(2) \cup \cdots\cup T(p)$ is
open. Moreover, since $T$ is compact, there exists a finite covering
of $T$ by open sets $U$ with the property that $T(p+1) \cup \cdots
\cup T(n)$ is given on $U$ as a finite union of closed subsets, each
defined by finitely many polynomial inequalities.
\end{par}
\end{proof}

\vspace{4mm}
\begin{cor}\label{Cor:NbhoodRetract}
Each subset $T(p)\subseteq T$ admits an open neighbourhood of which
$T(p)$ is a deformation retract.
\end{cor}

\begin{proof}

By Corollary \ref{Cor:CWDecomposition}, there exists a finite
simplicial complex $K$ and a homeomorphism $h:T\to K$, such that
each $T(p)$ is a union of sets $h^{-1}(F)$ for some relatively open
faces $F$ of $K$. But any union of relatively open faces on a
simplicial complex, finite or not, admits an open neighbourhood
which is a deformation retract, and we can transport these
neighbourhoods to $T$ via $h$.
\end{proof}

\vspace{4mm}
\begin{lem}\label{lemma-path}
Let $\gamma:[0,1]\rightarrow T(p)$ be a path. Then
$M(\gamma(0))=M(\gamma(1))$.
\end{lem}

\begin{proof}
Pick $t_0\in [0,1]$ and set $g_0 := \gamma(t_0)$. We have seen in
the proof of Lemma \ref{Lem:TildeSiLocallyClosed} that the inclusion
$M(g_0) \subseteq M(g)$ holds for all $g$ in some neighbourhood of
$g_0$, hence $M(\gamma(t_0)) \subseteq M(\gamma(t))$ holds for all
$t$ in some neighbourhood of $t_0$. But we assume $\dim M(\gamma(t))
= n-p$ for all $t$, so the equality $M(\gamma(t_0)) = M(\gamma(t))$
must hold for all $t$ close to $t_0$. The map $t\mapsto
M(\gamma(t))$ is therefore locally constant on $[0,1]$, hence
constant.
\end{proof}

\vspace{4mm}
\begin{lem}\label{key-lemma}
Let $M\subseteq V$ be a real linear subspace of dimension $p$ which
is generated by elements of $\Lambda$, and let $G_{M,\Lambda}
\subseteq G_\Lambda$ be the subgroup consisting of those $g\in
G_\Lambda$ satisfying $gM=M$. We have
$$\rank(G_{M,\Lambda}) < {\rm g.c.d.}(n,p).$$
\end{lem}

\begin{proof}
Consider the number field $k :=\Q[G_{M,\Lambda}]$ and set
$e:=[k:\Q]$. We can regard $\Lambda\otimes \Q$ and $(M\cap
\Lambda)\otimes \Q$ as $k_0$-vector spaces, and thus have
$$e \ | \ \dim_\Q (\Lambda\otimes \Q) =n \qquad \text{and} \qquad e \ | \ \dim_\Q ((M\cap \Lambda)\otimes \Q) =p$$
where the last equality holds since $M$ is rational. Therefore, $e$
divides ${\rm g.c.d.}(n,p)$. Since $G_{M,\Lambda}$ embeds into the
group of units $\cO_k^\ast$, we have $\rank(G_{M,\Lambda})\leq
\rank(\mathcal{O}_k^\ast)\leq e-1$, hence the claim.
\end{proof}

\vspace{4mm}
\begin{cor}\label{coro-rank}
For every $1\leq p \leq n$ and every $t\in T(p)$, the image of the
group homomorphism induced by the inclusion $T(p) \subseteq T =
G/G_\Lambda$
$$\rho:\pi_1(T(p),t)\rightarrow \pi_1(T,t)\cong G_\Lambda$$
has rank $\leq {\rm g.c.d.}(n,p)-1$.
\end{cor}

\begin{proof}
Choose $g\in G(p)$ in the class of $t$. We show that the image of
$\rho$ is contained in $G_{M(g),\Lambda}$. Since $M(g)$ is generated
by lattice elements and $\dim M(z)=n-p$ the desired conclusion
follows then from Lemma \ref{key-lemma}. Let $[\gamma]\in
\pi_1(T(p),t)$ be the class of a path $\gamma:[0,1]\rightarrow T(p)$
such that $\gamma(0)=\gamma(1)=t$. Lift it to a path
$\tilde{\gamma}:[0,1] \rightarrow G$ with $\tilde{\gamma}(0)=g$.
Setting $h:=\rho([\gamma])\in G_\Lambda$, we have $h
\tilde{\gamma}(0)=\tilde{\gamma}(1)$ by definition of the canonical
isomorphism $\pi_1(T,t) \cong G_\Lambda$. By Lemma \ref{lemma-path}
the equality $M(\tilde{\gamma}(0))=M(\tilde{\gamma}(1))$ holds,
hence
$$M(g)=M(hg)=h^{-1}M(g),$$
and hence $h\in G_{M(g),\Lambda}$ as claimed.
\end{proof}

\vspace{4mm}
\begin{proof}[Proof of Theorem \ref{thm-succ-minima}]
For $1 \leq p \leq n$, let $U_p$ be an open neighborhood of $T(p)$
such that $U_p$ is a retract of $T(p)$. In particular $T(p)=\vide$
if and only if $U_p=\vide$. Such neighborhoods exist by Corollary
\ref{Cor:NbhoodRetract}. We get an open covering $U_1 \cup \cdots
\cup U_n$ of $T$. Corollary \ref{coro-rank} implies that
\begin{equation}\label{Eqn:RankCondition}
\rank (H_1(W,\Z)\rightarrow H_1(T,\Z))< p
\end{equation}
holds for every connected component $W$ of $U_p$. It follows from
\eqref{Eqn:RankCondition} and theorem \ref{thm-covering} that the
sets $U_1, \ldots, U_{r+s-1}$ do not cover $T$, hence $U_p$ and
hence $T(p)$ must be nonempty for some $p\geq r+s$. This implies
that there exists $g\in G$ with $\dim M(g) \leq s$, which was to be
shown.
\end{proof}

\vspace{14mm}
\section{Proof of the Main Theorem}\label{secproof}

\begin{par}
Equip the vector space $V=\IR^r \oplus \IC^{s}$ with the scalar
product and the norm map introduced in
\ref{Para:IntroLatticeOfSignature}, and let $G$ be the group of diagonal matrices
$g = {\rm diag}(g_1,\dots,g_{r+s})$ with positive real entries $g_i$
such that $$g_1\dots g_r (g_{r+1}\dots g_{r+s})^2 = 1.$$
 Theorem \ref{thm-final-new} below gives an
upper bound on the inhomogeneous minimum of every lattice $\Lambda
\subseteq V$ whose $G$--orbit is compact. Our main Theorem stated in
the introduction is a consequence of it.
\end{par}

\vspace{4mm}
\begin{thm}\label{thm-final-new}
Let $\Lambda$ be a lattice in $V$ such that $G\Lambda$ is compact.
The following inequality holds for every $1\leq a\leq r+s$.
$$m(\Lambda)^s\cdot M(\Lambda)^a\:\:\leq\:\: \big ( 2^{s-a}\cdot\gamma_n^{s+a} \cdot n^{-s}\big)^\frac{n}{2}\cdot \det(\Lambda)^{s+a}$$
\end{thm}

\begin{proof}
By Theorem  \ref{thm-succ-minima} there exists $g\in G$ such that
$\mu_{s+1}(g\Lambda) =\cdots = \mu_{n}(g\Lambda)$ holds. We have
then
\begin{equation}\label{Eqn:ProofMain2-new}
\mu_1(g\Lambda)^s\cdot \mu_n(g\Lambda)^{a}\leq \mu_1(g\Lambda)\cdot
\cdots\cdot \mu_{s}(g\Lambda)\cdot \mu_{n}(g\Lambda)^{a}\leq
\gamma_n^{\frac{s+a}{2}}\cdot \det(g\Lambda)^{\frac{s+a}{n}}
\end{equation}
by Lemma \ref{Mink} with $t=s+a$. Furthermore, we have
\begin{equation}\label{Eqn:ProofMain1-new}
\mu_n(g\Lambda)\geq \sqrt{2}\cdot
M(g\Lambda)^{\frac{1}{n}}=\sqrt{2}\cdot M(\Lambda)^{\frac{1}{n}}
\end{equation}
by Lemma \ref{lemma-bound-inhom}, and
\begin{equation}\label{Eqn:ProofMain3-new}
\mu_1(g\Lambda)\geq \frac{\sqrt n}{\sqrt 2} \cdot m(g\Lambda)^{\frac
1 n}=\frac{\sqrt n}{\sqrt 2} \cdot m(\Lambda)^{\frac 1 n}
\end{equation}
by Lemma \ref{lemma-bound-hom}. The statement of the theorem follows
by combining \eqref{Eqn:ProofMain2-new}, \eqref{Eqn:ProofMain1-new}
and \eqref{Eqn:ProofMain3-new}.
\end{proof}

\vspace{4mm}
\begin{para}\label{Para:CompactnessForNumberFieldLattice}
\begin{par}
Let $K$ be a number field of degree $n=r+2s$ over $\IQ$, with $r$
real embeddings $\sigma_1,\ldots,\sigma_r$ and $s$ non-conjugated
complex embeddings $\sigma_{r+1},\ldots,\sigma_{r+s}$. Let $d_K$ be
the absolute value of the discriminant of $K$. From the chosen
ordering of the embeddings of $K$ we obtain a $\IQ$--linear map
$\sigma: K \to V$ sending $x\in K$ to
$(\sigma_1(x),\ldots,\sigma_{r+s}(x))$. The image of $\cO_K$ under
this map is a lattice $\Lambda := \sigma(\cO_K) \subseteq V$ of
volume $2^{-s}\cdot \sqrt{d_K}$ (see for example \cite{Samuel}, page
57).
\end{par}

\bigskip

\begin{par}
Let $\tilde G$ denote the group of diagonal matrices $g = {\rm diag}(g_1,\dots,g_{r+s})$ with $g_1,\dots,g_r \in \IR^{*}$
and $g_{r+1},\dots, g_{r+s} \in \IC^*$ satisfying $|g_1 \dots g_r (g_{r+1} \dots g_{r+s})^2| = 1$. The group $\tilde G$ 
contains $G$ as a direct factor. Indeed, we have $\tilde G = G \times U$, where $U = (\IZ/ 2 \IZ)^r \times (S^1)^s$.
The action of $G$ on $V$ extends to an action of $\tilde G$ in the obvious way. 

\medskip

\noindent
{\bf Claim.} {\it  $G \Lambda$ is compact. }

\medskip
\noindent
{\it Proof.} It suffices to prove that $\tilde G \Lambda$ is compact. Let $\epsilon:\cO_K^\ast \to G$ be the group homomorphism defined by $\epsilon(x) = {\rm diag}(\sigma_1(x),\dots,\sigma_{r+s}(x))$. 
The equality $\epsilon(x)\sigma(y) = \sigma(xy)$ holds for all $x \in \cO_K^*$ and all $y \in K$. This shows that the image of $\epsilon$
stabilizes the lattice $\Lambda = \sigma(\cO_K)$. In other words, $\epsilon (\cO_K^*)$ is contained in $\tilde G_{\Lambda}$. We have
$\tilde G \Lambda = \tilde G / \tilde G_{\Lambda}$, hence to show that $\tilde G \Lambda$ is compact, it is enough to prove that
$\tilde G / \epsilon(\cO_K^*)$ is compact. This follows from (the proof of) Dirichlet's unit theorem. Indeed, let us define
$L : \tilde G \to \IR^{r+s}$ by ${\rm diag}(g_1,\dots,g_{r+s}) \mapsto ({\rm log}|g_1|,\dots,{\rm log}|g_{r+s}|)$. The kernel of $L$
is the group $U$, which is compact, and the image of $\epsilon (\cO_K^*)$ by $L$ is a lattice in $\IR^{r+s-1}$, hence cocompact. 

\end{par}

\end{para}

\vspace{4mm}
\begin{thm}[Main Theorem]\label{thm:main}
Let $K$ be a number field of signature $(r,s)$ and degree $n =
r+2s$, and let $d_K$ be the absolute value of the discriminant of
$K$. Then
$$ M(K) \:\:\leq\:\:  2^\frac{-s(s+a)}{a}  \cdot \big ( 2^{s-a}\cdot\gamma_n^{s+a} \cdot n^{-s}\big)^\frac{n}{2a}\cdot {d_K}^{\frac{s+a}{2a}}$$
holds for every $1\leq a \leq r+s$.
\end{thm}

\begin{proof}
By the Claim, 
the lattice $\Lambda := \sigma(\cO_K)\subseteq V$ associated with
$K$ has a compact $G$--orbit. Since we have $m(\Lambda)=1$, the
inequality follows directly from Theorem \ref{thm-final-new}.
\end{proof}

\vspace{4mm}
\begin{para}
For number fields with a large discriminant, the choice $a=r+s$ will
give the best upper bound in Theorem \ref{thm:main}, whereas for
number fields with a small discriminant also other choices for $a$
can be interesting. For $a=r+s$ and using the estimate $\gamma_n\leq
\frac{n}{2}$ for $n\geq 4$ (Theorem 2.7.4. of \cite{MR1957723} and
page 17 of \cite{MR0506372}) we obtain the theorem stated in the
introduction.
\end{para}

\vspace{4mm}
\begin{rem}
For small degrees $n$, the exact value of the Hermite constant
$\gamma_n$ is known. The following table presents the explicit
bounds that we obtain from Theorem \ref{thm:main} for $n\leq 5$. To
the authors knowledge, these bounds are the best known for $n=4,5$
and $s\neq 0$.

\vspace{4mm}
\begin{center}
\begin{tabular}{|c|c|l|}
  \hline
  $n$ & $s$ & Upper bound for $M(K)$ \\
  \hline
  $1$ & $0$ & $\frac{1}{\sqrt{2}}\cdot \sqrt{d_K}$ \\
  $2$ & $0$ & $\frac{1}{\sqrt{3}}\cdot \sqrt{d_K}$ \\
  $2$ & $1$ & $\frac{1}{6}\cdot d_K$ \\
  $3$ & $0$ & $\frac{1}{2}\cdot \sqrt{d_K}$ \\
  $3$ & $1$ & $\min\left(\frac{1}{6\sqrt{3}}\cdot d_K,\frac{1}{2\sqrt[4]{108}}\cdot d_K^{\frac{3}{4}}\right)$ \\
  $4$ & $0$ & $\frac{1}{2}\cdot \sqrt{d_K}$ \\
  $4$ & $1$ & $\min\left(\frac{1}{16}\cdot d_K,\frac{1}{8}\cdot d_K^{\frac{3}{4}},\frac{1}{4\sqrt[3]{4}}\cdot d_K^{\frac{2}{3}}\right)$ \\
  $4$ & $2$ & $\min\left(\frac{1}{512}\cdot d_K^{\frac{3}{2}},\frac{1}{64}\cdot d_K\right)$ \\
  $5$ & $0$ & $\frac{1}{2}\cdot \sqrt{d_K}$ \\
  $5$ & $1$ & $\min\left(\frac{2}{25\sqrt{5}}\cdot d_K,\frac{1}{4\sqrt[4]{20}}\cdot d_K^{\frac{3}{4}},\frac{1}{2\sqrt[6]{3125}}\cdot d_K^{\frac{2}{3}},\frac{1}{2\sqrt[8]{12500}}\cdot d_K^{\frac{5}{8}}\right)$ \\
  $5$ & $2$ & $\min\left(\frac{2}{3125}\cdot d_K^{\frac{3}{2}},\frac{1}{50\sqrt{5}}\cdot d_K,\frac{1}{10\sqrt[3]{100}}\cdot d_K^{\frac{5}{6}}\right)$ \\
  \hline
\end{tabular}
\end{center}
\end{rem}

\vspace{5mm}
\section*{Acknowledgements}
We warmly thank Curtis McMullen for his valuable comments on a
previous version of this paper.

\vspace{5mm}

\bibliographystyle{elsarticle-num}

\end{document}